\documentclass[12pt, a4paper]{article}
\usepackage{graphicx}
\usepackage{amssymb,latexsym,amsmath}
\usepackage{epstopdf}
\usepackage{amsthm}
\usepackage{array,enumerate}
\usepackage{indentfirst}				
\usepackage{url,color}			

\newtheorem{lemma}{Lemma}[section]
\newtheorem{theorem}[lemma]{Theorem}

\newtheorem{corollary}[lemma]{Corollary}

\newcommand{\forme}[1]{}

\begin{document}

\date{\today}
\title{Light tails and the Hermitian dual polar graphs}
\author{
{\bf Jack Koolen}\footnote{JHK was partially supported by the National Natural Science Foundation of China
(No. 11471009). }\\School of Mathematical Sciences\\ University of Science and Technology of China, \\Wen-Tsun Wu Key Laboratory of the Chinese Academy of Sciences, \\
230026, Anhui, PR China \\
e-mail: koolen@ustc.edu.cn
\\
\\
{\bf Zhi Qiao}\\School of Mathematical Sciences\\ University of Science and Technology of China, \\
230026, Anhui, PR China \\
email: gesec@mail.ustc.edu.cn}
\maketitle
{\bf Dedicated to Andries Brouwer on the occasion of his 65th birthday}
\begin{abstract}

Juri\'{s}i\v{c} et al. \cite{JTZ10} conjectured (see also \cite[Problem 13]{DKT}) that if a distance-regular graph $\Gamma$ with diameter $D$ at least three has a light tail, then one of the following holds:
\begin{enumerate}
\item $a_1 =0$;
\item $\Gamma$ is an antipodal cover of diameter three;
\item $\Gamma$ is tight;
\item $\Gamma$ is the halved $2D+1$-cube;
\item $\Gamma$ is a Hermitian dual polar graph $^2A_{2D-1}(r)$ where $r$ is a prime power. 
\end{enumerate}

In this note, we will consider the case when the light tail corresponds to the eigenvalue $-\frac{k}{a_1 +1}$. Our main result is:\\
{\bf Theorem} Let $\Gamma$ be a non-bipartite distance-regular graph with valency $k \geq 3$ , diameter $D \geq 3$ and distinct eigenvalues $\theta_0 > \theta_1 > \cdots > \theta_D$. Suppose that $\Gamma$ is $2$-bounded with smallest eigenvalue  $\theta_D = -\frac{k}{a_1 +1}$. If the minimal idempotent $E_D$, corresponding to eigenvalue $\theta_D$, is a light tail, then $\Gamma$ is the dual polar graph $^2A_{2D-1}(r)$, where $r$ is a prime power.
\\
\\
 As a consequence of this result we will also show:\\
{\bf Theorem.}
 Let $\Gamma$ be a distance-regular graph with valency $k \geq 3$, diameter $D \geq 2$, $a_1 =1$ and $\theta_0 > \theta_1 > \cdots > \theta_D$. If $c_2  \geq5$ and $\theta_D = -k/2$, then $c_2 =5$ and $\Gamma$ is the dual polar graph $^2A_{2D-1}(2)$.

\end{abstract}
\section{Introduction}
For definitions, see next section and \cite{BCN}. Let $\Gamma$ be a distance-regular graph with $n$ vertices, valency $k \geq 3$ and diameter $D \geq 2$. A minimal idempotent E of $\Gamma$, say corresponding to eigenvalue $\theta$, is called a light tail if $E \circ E = a E_0 + b F$, where $E_0 = \frac{1}{n}J$, $F \neq E_0$ is a minimal idempotent of $\Gamma$, and $a,b$ are non-zero real numbers. 

Juri\'{s}i\v{c} et al. \cite{JTZ10} conjectured (see also \cite[Problem 13]{DKT}) that if a distance-regular graph $\Gamma$ with diameter $D$ at least three has a light tail, then one of the following holds:
\begin{enumerate}
\item $a_1 =0$;
\item $\Gamma$ is an antipodal cover of diameter three;
\item $\Gamma$ is tight;
\item $\Gamma$ is the halved $2D+1$-cube;
\item $\Gamma$ is a Hermitian dual polar graph $^2A_{2D-1}(r)$ where $r$ is a prime power. 
\end{enumerate}

Juri\'{s}i\v{c} et al.  \cite{JTZ10} showed that a distance-regular graph $\Gamma$ with diameter at least three has a light tail corresponding to $-1$ if and only if $\Gamma$ is an antipodal cover of diameter three. 
If a tight distance-regular graph has a light tail, then it  is easy to show that $\Gamma$ is antipodal (using the fact that  tight implies that $a_D =0$.).
In this note, we will consider  distance-regular graphs with valency $k$ and intersection number $a_1$ having a light tail corresponding to the eigenvalue $-\frac{k}{a_1 +1}$. Our main result is:
\begin{theorem}\label{dlp}
Let $\Gamma$ be a non-bipartite distance-regular graph with valency $k \geq 3$ , diameter $D \geq 3$ and distinct eigenvalues $\theta_0 > \theta_1 > \cdots > \theta_D$. Suppose that $\Gamma$ is $2$-bounded with smallest eigenvalue  $\theta_D = -\frac{k}{a_1 +1}$. If the minimal idempotent $E_D$, corresponding to eigenvalue $\theta_D$, is a light tail, then $\Gamma$ is the dual polar graph $^2A_{2D-1}(r)$, where $r$ is a prime power.
\end{theorem}
 As a consequence of this result we will also show:
 \begin{theorem}
 Let $\Gamma$ be a distance-regular graph with valency $k \geq 3$, diameter $D \geq 2$, $a_1 =1$ and distinct eigenvalues $\theta_0 > \theta_1 > \cdots > \theta_D$. If $c_2  \geq 5$ and $\theta_D = -k/2$, then $c_2 =5$ and $\Gamma$ is the dual polar graph $^2A_{2D-1}(2)$. 
 \end{theorem}
 
The authors \cite{KQ15} conjectured that if $\Gamma$ is a distance-regular graph with valency $k\geq 3$, $a_1 =1$, diameter $D \geq 3$ and smallest eigenvalue $-k/2$, then $\Gamma$ is either the Hamming graph $H(D,3)$, the dual polar graph $B_D(2)$ or the dual polar graph $^2A_{2D-1}(2)$ if $D$ is large enough. We confirm this conjecture when $c_2 \geq 5$. 

This note is organized as follows. In the next section, we give the definitions and some preliminary results.  
Then, in Section 3, we start with a sufficient condition for a distance-regular graph to have light tail,  and then give a proof of Theorem 1.1. In the last section we will give a proof of Theorem 1.2.

\section{Preliminaries and definitions}

All graphs considered in this note are finite, undirected and simple (for more background information and undefined notions, see \cite{BCN} or \cite{DKT}). 
We denote $x\sim y$ if the vertices $x,y$ are adjacent. 
For a connected graph $\Gamma=(V(\Gamma),E(\Gamma))$, the {\em distance} $d_{\Gamma}(x,y)$ between any two vertices $x,y$ is the length of a shortest path between $x$ and $y$ in $\Gamma$, and the {\em diameter} $D(\Gamma)$ is the maximum distance between any two vertices of $\Gamma$, and if the graph is clear from the context, we use $d(x,y)$ and $D$ for simplicity. 
For any vertex $x$, let $\Gamma_i(x)$ be the set of vertices in $\Gamma$ at distance precisely $i$ from $x$, where $0 \leq i\leq D$. 
A graph $\Gamma$ is locally $\Delta$, when for all vertices $x$ the subgraph induced by $\Gamma_1(x)$ is isomorphic to $\Delta$.

A connected graph $\Gamma$ with diameter $D$ is called a {\em distance-regular graph} if there are integers $b_i$, $c_i$ ($0\leq i \leq D$) such that for any two vertices $x,y\in V(\Gamma)$ with $d(x,y)=i$, there are exactly $c_i$ neighbors of $y$ in $\Gamma_{i-1}(x)$ and $b_i$ neighbors of $y$ in $\Gamma_{i+1}(x)$. 
The numbers $b_i, c_i$ and $a_i:=b_0-b_i-c_i~(0 \leq i\leq D)$ are called the {\em intersection numbers} of $\Gamma$. 
Set $c_0=b_{D}=0$. We observe $a_0=0$ and $c_1=1$. 
The array $\iota(\Gamma)=\{b_0,b_1,\ldots,b_{D-1};c_1,c_2,\ldots,c_{D}\}$ is called the {\em intersection array} of $\Gamma$. 
In particular, $\Gamma$ is a regular graph with valency $k:=b_0$. 

Let $\Gamma$ be a distance-regular graph with $n$ vertices,  valency $k$ and diameter $D$. 
For $0\leq i\leq D$, let  $A_i$ be the $\{0, 1\}$-matrix indexed by the vertices of $\Gamma$ with $(A_i)_{xy}= 1$ if and only $d(x, y)=i$ for vertices $x, y$ of $\Gamma$. 
We call $A_i$ the {\em distance-$i$ matrix} of $\Gamma$, and $A:=A_1$ the {\em adjacency matrix} of $\Gamma$. 
The {\em eigenvalues} $\theta_0>\theta_1>\cdots>\theta_D$ of $\Gamma$ are just the eigenvalues of its adjacency matrix $A$. 
We denote $m_i$ the {\em multiplicity} of $\theta_i$. 

For each eigenvalue $\theta_i$ of $\Gamma$, let $W_i$ be a matrix with its columns forming an orthonormal basis for the eigenspace of $\theta$. 
And $E_i:=W_i W_i^T$ is called the {\em minimal idempotent} corresponding to the eigenvalue $\theta_i$, satisfying $E_i E_j= \delta_{ij} E_j$ and $A E_i=\theta_i E_i$, where $\delta_{ij}$ is the Kronecker symbol. 
Note that $nE_0$ is the all-one matrix $J$. 

The set of distance matrices $\{A_0=I,A_1,A_2,\ldots,A_D\}$ forms a basis for a commutative $\mathbb R$-algebra $\mathcal A$, known as the {\em Bose-Mesner algebra}. 
The set of minimal idempotents $\{E_0=\frac{1}{n}J,E_1,E_2,\ldots,E_D\}$ is another basis for $\mathcal A$. 
There exist $(D+1)\times(D+1)$ matrices $P$ and $Q$ (see \cite[p.45]{BCN}), such that the following relations hold:
    \begin{equation}	\label{eq:pq}
    A_i=\sum_{j=0}^D P_{ji}E_j ~~ \text{and} ~~ E_i=\frac{1}{n}\sum_{j=0}^D Q_{ji}A_j~(0\leq i\leq D).
    \end{equation}
Note that  $Q_{0i}=m_i$ (see \cite[Lemma 2.2.1]{BCN}). 
The Bose-Mesner algebra $\mathcal A$ is also closed under the Schur product, therefore we also have
    \begin{equation}	\label{eq:krn}
    E_i\circ E_j=\frac{1}{n}\sum_{h=0}^D q_{ij}^h E_h ~(0\leq i,j\leq D),
    \end{equation}
where the numbers $q_{ij}^h~(0\leq i,j,h\leq D)$ are called the {\em Krein parameters} (see \cite[p.48]{BCN}).
It is well known, cf. \cite[Lemma 2.3.1, Theorem 2.3.2]{BCN}, the Krein parameters are non-negative real numbers and also $q_{ij}^0=\delta_{ij}m_j$ holds. 

Let $E_i=W_i W_i^T$ be a minimal idempotent corresponding to the eigenvalue $\theta_i$, where the columns of $W_i$ form an orthonormal basis of the eigenspace of $\theta_i$. 
We denote the $x$-th row of $W_i$ by $\hat{x}$. 
Note that $E_i\circ A_j=\frac{1}{n}Q_{ji}A_j$, hence all the vectors $\hat{x}$ has the same length $\frac{1}{n}Q_{0i}$ and the cosine of the angle between two vectors $\hat{x}$ and $\hat{y}$ is $u_j(\theta_i):=\frac{Q_{ji}}{Q_{0i}}$, where $d(x,y)=j$.
The map $x\mapsto \hat{x}$ is called a {\em spherical representation} of $\Gamma$, and the sequence $\{u_j(\theta_i)\mid 0\leq j\leq D\}$ is called the {\em standard sequence} of $\Gamma$,  corresponding to the eigenvalue $\theta_i$. 
As $A W=\theta_i E_i W=\theta_i W$, we have $\theta_i \hat{x}=\sum_{y\sim x} \hat{y}$, and hence the following holds:
    \begin{equation}	\label{eq:std}
    	\begin{aligned}
   	 		&c_j u_{j-1}(\theta_i)+ a_j u_j(\theta_i)+ b_j u_{j+1}(\theta_i)=\theta_i u_j(\theta_i)~(1\leq j\leq D-1)\\
			&c_Du_{D-1}(\theta_i)+a_D u_{D}(\theta_i)=\theta_i u_D(\theta_i)
		\end{aligned}	
    \end{equation}
with $u_0(\theta_i)=1$ and $u_1(\theta_i)=\frac{\theta_i}{k}$. 

A minimal idempotent $E$ is called a {\em light tail} if $E\circ E=aE_0 + b F$ for some minimal idempotent $F\neq E_0$ and nonzero real numbers $a,b$. 
The minimal idempotent $F$ is called the {\em associated minimal idempotent}  to the light tail $E$. 
Note that $E_0\circ E_0=\frac{1}{n} E_0$ holds and that if $\Gamma$ is bipartite, then also $E_D\circ E_D=\frac{1}{n} E_0$holds, see for example \cite[Lemma 3.3]{P99}. So the eigenvalue corresponding to the light tail $E$ is different from $\pm k$. 

Juri\v{s}i\'{c} et al. \cite[Theorem 3.2]{JTZ10} has shown the following result. 
\begin{lemma}	\label{mbd}
Let $\Gamma$ be a distance-regular graph with valency $k\geq 3$, diameter $D\geq 2$ and distinct eigenvalues $\theta_0>\theta_1>\cdots>\theta_D$. Let $\theta_i\neq \pm k$ be an eigenvalue of $\Gamma$, then 
    \begin{equation}
    \frac{m_i-k}{k}\geq -\frac{(\theta_i+1)^2 a_1(a_1+1)}{((a_1+1)\theta_i+k)^2+ka_1 b_1},
    \end{equation}
and equality holds if and only if $E_i$ is a light tail. 
\end{lemma}

For a graph $\Gamma$, a partition $\Pi=\{P_1,P_2,\ldots ,P_t\}$ of $V(\Gamma)$ is called {\em equitable} if there are constants $\alpha_{ij}$ $(1\leq i,j\leq t)$ such that all vertices $x\in P_i$ have exactly $\alpha_{ij}$ neighbours in $P_j$. 
The $\alpha_{ij}$'s $(1\leq i,j\leq t)$ are called the parameters of the equitable partition. 

Let $\Gamma$ be a distance-regular graph. 
For a set $S$ of vertices of $\Gamma$, define $S_i=\{x\in V(\Gamma)\mid d(x,S)=i \}$, where $d(x,S):=\min\{d(x,y)\mid y\in S\}$. 
Then number $\rho(S):=\max\{i\mid S_i\neq \emptyset\}$ is called the {\em covering radius} of $S$. The set $S$ is called a {\em completely regular code} of $\Gamma$ if the distance-partition $\{S=S_0, S_1,S_2,\ldots, S_{\rho(S)}\}$ is equitable. 

The following result was first shown by Delsarte \cite{D73} for strongly regular graphs, and extended by Godsil \cite{G93} to distance regular graphs. 

\begin{lemma}\label{clique}
    Let $\Gamma$ be a distance-regular graph with valency $k$ and smallest eigenvalue $\theta_{\min}$. Let $C$ be a clique in $\Gamma$ with  $c$ vertices. Then $c\leq 1-\frac{k}{\theta_{\min}}$, and equality holds if and only if $C$ is a completely regular code with covering radius $D-1$.     
\end{lemma}

A clique $C$ in a distance-regular graph $\Gamma$ that attains the above bound is called a {\em Delsarte clique}. 
A distance-regular graph $\Gamma$ is called {\em geometric} (with respect to $\mathcal C$) if it contains a collection $\mathcal C$ of Delsarte cliques such that each edge is contained in a unique $C\in \mathcal C$. 

Let $\Gamma$ be a geometric distance-regular graph with respect to $\mathcal C$. 
As each edge is contained in a unique Delsarte clique in $C$, we see the Delsarte cliques have size $a_1+2$ and $\theta_{\min}=-\frac{k}{a_1+1}$, where $\theta_{\min}$ is the smallest eigenvalue of $\Gamma$. 
And the parameter $\gamma_i:=|\Gamma_i(x)\cap C|$ $(0\leq i\leq D)$ is well-defined (see \cite[Theorem 7.1]{G93}), where $C$ is any Delsarte clique of $\Gamma$ and $x$ is any vertex with $d(x,C)=i$. 
\begin{lemma}(cf. \cite[Proposition 4.3]{DKT})\label{lm:1}
Let $\Gamma$ be a geometric distance-regular graph with valency $k\geq 3$, diameter $D\geq 2$ and distinct eigenvalues $\theta_0>\theta_1>\cdots>\theta_D$. Then the following holds: 
	\begin{align}
    	\gamma_i u_i(\theta_D)+(a_1+2-\gamma_i) u_{i+1}(\theta_D)=0 ~~~~(0\leq i\leq D-1). \label{eq:f}
	\end{align}
where $u_i(\theta_D)$ $(0\leq i\leq D)$ is the standard sequence corresponding to $\theta_D$. 
\end{lemma}

A subgraph $\Delta$ of a graph $\Gamma$ is called {\em strongly closed} if, for all vertices $x,y\in V(\Delta)$ and $z\in V(\Gamma)$ such that $d_{\Gamma}(x,z)+d_{\Gamma}(z,y)\leq d_{\Gamma}(x,y)+1$, holds, then we have that $z\in V(\Delta)$.  
A graph $\Gamma$ is said to be {\em $m$-bounded} for some $1\leq m\leq D$, if for all $1\leq i\leq m$ and all vertices $x,y$ at distance $i$, there exists a strongly closed subgraph $\Delta(x,y)$ of $\Gamma$ with diameter $i$, containing $x$ and $y$ as vertices. 

The following result was shown by Hiraki \cite[Theorem 1.1]{H99}, see also \cite[Proposition 11.3]{DKT}.
\begin{lemma}\label{hiraki}
Let $\Gamma$ be a distance-regular graph with diameter $D\geq 3$. 
Let $m\in\{1,2,\ldots, D-1\}$ and $c_{m+1}\neq 1$. 
Then $\Gamma$ is $m$-bounded if and only if $\Gamma$ is $K_{1,1,2}$-free, $a_1\neq 0$, $a_i=c_i a_1$ holds for $i=1,2,\ldots, m$ and $c_{m-1}<c_m$. 
\end{lemma}

A generalized quadrangle of order $(s,t)$, denoted by GQ$(s, t)$, (where $s, t$ are positve integers), is a $K_{2,1,1}$-free distance-regular graph with diameter 2 with intersection array $\{(t+1)s, ts; 1, t+1\}$.

\section{Characterization of the Hermitian dual polar graphs}
We will give a characterization of the Hermitian dual polar graphs $^2A_{2D-1}(r)$, where $r$ is a prime power.
First we give a sufficient condition for a minimal idempotent to be a light tail of a distance-regular graph.
\begin{lemma}		\label{lbd}
Let $\Gamma$ be a distance-regular graph with $n$ vertices, valency $k\geq 3$, diameter $D\geq 2$, $a_1\neq 0$ and distinct eigenvalues $\theta_0>\theta_1>\cdots>\theta_D=-\frac{k}{a_1+1}$. 
Then $\theta_1\geq \frac{k-(a_1+1)(a_1 +2)}{(a_1+1)^2}$ and equality holds if and only if $E_D$ is a light tail with associated minimal idempotent $E_1$.
\end{lemma}
\begin{proof}
For each eigenvalue $\theta_i$ of $\Gamma$, we denote $m_i$ its multiplicity and $u_j(\theta_i)~(0\leq j\leq D)$ its standard sequence. 

By the definition of standard sequence, we have 
    \begin{equation}	\label{eq:a}
    \frac{n}{m_i}E_i=\sum_{j=0}^D u_j(\theta_i)A_j ~(0\leq i\leq D).
    \end{equation}
As $A_i\circ A_j=\delta_{ij} A_j$, the following holds
    \begin{equation}	\label{eq:b}
    (\frac{n}{m_D}E_D)\circ(\frac{n}{m_D}E_D)=\sum_{i=0}^D u_i(\theta_D)^2 A_i. 
    \end{equation}
By (\ref{eq:krn}) and (\ref{eq:a}), we see
    \begin{align}
    (\frac{n}{m_D}E_D)\circ (\frac{n}{m_D}E_D)
    	&=\frac{n^2}{m_D^2} \cdot \frac{1}{n}\sum_{i=0}^D q_{DD}^i E_i \\
    	&=\frac{n}{m_D}E_0+\frac{n}{m_D^2}\sum_{i=1}^D q_{DD}^i E_i\\
    	&=\frac{1}{m_D}J+\sum_{j=0}^D (\sum_{i=1}^D\frac{m_iq_{DD}^i}{m_D^2} u_j(\theta_i))\cdot A_j\\
    	&=\frac{1}{m_D}J+\frac{m_D-1}{m_D}\tilde{E}, 	\label{eq:c}
    \end{align}
where $\tilde{E}=\sum_{j=0}^D \tilde{u}_j A_j$, $\tilde{u}_j=\sum_{i=1}^D\alpha_i u_j(\theta_i)$ and $\alpha_i=\frac{m_iq_{DD}^i}{m_D(m_D-1)}\geq 0$.

Since $J=\sum_{i=0}^D A_i$, by (\ref{eq:b}) and (\ref{eq:c}), we have 
    \begin{equation}
    u_i^2(\theta_D)=\frac{1}{m_D}+\frac{m_D-1}{m_D}\tilde{u}_i ~(0\leq i\leq D).
    \end{equation}
Since $u_0(\theta_D)=1$ and $u_1(\theta_D)= \frac{\theta_D}{k}= -\frac{1}{a_1+1}$, we see that $\tilde{u}_0=1$ and $\tilde{u}_1=\frac{m_D(\frac{1}{a_1+1})^2-1}{m_D-1}$. 
As $a_1\neq 0$, we see $\theta_D=-\frac{k}{a_1+1}\neq -k$. 
And Lemma \ref{mbd} implies $m_D\geq \frac{a_1 k}{a_1+1}+1$. 
Hence we have $\tilde{u}_1\geq \frac{k-(a_1+1)(a_1+2)}{(a_1+1)^2 k}$. 
Note that $1=\tilde{u}_0=\sum_{i=1}^D\alpha_i$ and $\tilde{u}_1=\sum_{i=1}^D\alpha_i u_1(\theta_i)=\frac{1}{k}\sum_{i=1}^D\alpha_i\theta_i$. 
We have that $\theta_1\geq k \tilde{u}_1\geq \frac{k-(a_1+1)(a_1+2)}{(a_1+1)^2}$, and equality holds if and only if  $\tilde{E} = E_1$ and 
$$ (\frac{n}{m_D}E_D)\circ (\frac{n}{m_D}E_D)=\frac{1}{m_D}J+\frac{m_D-1}{m_D}E_1,$$ both hold, by (\ref{eq:c}) 
This means that  equality holds if and only if 
 $E_D$ is a light tail with associated minimal idempotent $E_1$, by the definition of a light tail.  
\end{proof}

The following result is a consequence of \cite[Cor 3.4]{BKP15}.
\begin{lemma}		\label{ubd}
Let $\Gamma$ be a distance-regular graph with valency $k\geq 3$, diameter $D\geq 2$ and distinct eigenvalues $\theta_0>\theta_1>\cdots>\theta_D$.  
If $c_2\neq 1$ and $\Gamma$ contains an induced $GQ(s,c_2-1)$ for some positive integer $s$, then $\theta_1\leq \frac{k-(a_1+1)}{c_2-1}-1$. 
\end{lemma}
\begin{proof}
As $\Gamma$ contains an induced subgraph $\Gamma'=GQ(s,c_2-1)$, then, by \cite[Corollary 3.4]{BKP15}, we see 
$-c_2=\theta_{\min}(\Gamma')\geq -1-\frac{b_1}{\theta_1+1}$. 
The result follows.
\end{proof} 

The following result  follows immediately from Lemmata \ref{lbd} and \ref{ubd}. 
\begin{theorem}		\label{lt}
Let $\Gamma$ be a distance-regular graph with valency $k\geq 3$, diameter $D\geq 2$, $a_1\neq  0$ and distinct eigenvalues $\theta_0>\theta_1>\cdots>\theta_D=-\frac{k}{a_1+1}$. 
If $\Gamma$ contains an induced $GQ(a_1+1,c_2-1)$ with $c_2=(a_1+1)^2+1$, then $\theta_1=\frac{k-(a_1+1)(a_1 +2)}{(a_1+1)^2}$ and $E_D$ is a light tail with associated minimal idempotent $E_1$.
\end{theorem}
\begin{proof}
Since $c_2=(a_1+1)^2+1\neq 1$ and $\Gamma$ contains an induced $GQ(a_1+1,c_2-1)$, we have $\theta_1\leq \frac{k-(a_1+1)}{c_2-1}-1=\frac{k-(a_1+1)}{(a_1+1)^2}-1$, by Lemma \ref{ubd}. 
Lemma \ref{lbd} implies $\theta_1\geq \frac{k-(a_1+1)}{(a_1+1)^2}-1$. 
These two bounds on $\theta_1$ imply that $\theta_1=\frac{k-(a_1+1)}{(a_1+1)^2}-1$. 
Now again by Lemma \ref{lbd} it follows that $E_D$ is a light tail with associated minimal idempotent $E_1$. 
\end{proof}

{\bf Proof of Theorem 1.1.} 
Since $E_D$ is a light tail with $\theta_D=-\frac{k}{a_1+1}$,  the graph  $\Gamma$ is locally the disjoint union of $(a_1+1)$-cliques by \cite[Corollary 6.3]{JTZ10}. 
This implies that all maximal cliques of $\Gamma$ are $(a_1+2)$-cliques, and hence Delsarte cliques by Lemma \ref{clique}. It follows that $\Gamma$ is geometric.  

Note that we have the following equations: 
	\begin{align}
        &c_i+a_i+b_i=k \label{eq:a1},& &(0\leq i\leq D)\\
        &c_i u_{i-1}(\theta_D)+a_iu_i(\theta_D)+b_i u_{i+1}(\theta_D)=\theta_D u_i(\theta_D) \label{eq:a2},& &(1\leq i\leq D-1)\\
        &\gamma_i u_i(\theta_D)+(a_1+2-\gamma_i) u_{i+1}(\theta_D)=0 \label{eq:a3},& &(0\leq i\leq D-1)
    \end{align}
where Equation (\ref{eq:a3}) is from Lemma \ref{lm:1}. 

As $u_i(\theta_D) \neq 0$ for $0 \leq i \leq D$ by \cite[Corollary 4.1.2]{BCN} we see that $\gamma_i \leq a_1 +1$ for $0\leq i \leq D-1$.
Then Equations (\ref{eq:a1}) (\ref{eq:a2}) and (\ref{eq:a3}) implies the following:
	\begin{align}		
     a_i  	&=c_i\frac{a_1+1-\gamma_{i-1}}{\gamma_{i-1}}+b_i\frac{\gamma_i-1}{a_1+1-(\gamma_i-1)}&	&(1\leq i\leq D-1)\label{eq:a4}\\
    		&=c_i\frac{a_1+1-\gamma_{i-1}-(\gamma_i-1)}{\gamma_{i-1}}+k\frac{\gamma_i-1}{a_1+1}& &(1\leq i\leq D-1)\label{eq:a5}
    \end{align} 

Now we assume $\Gamma$ is $m$-bounded for some $2\leq m\leq D-2$. We will show that $\Gamma$ is $(m+1)$-bounded, and this implies, by induction, that  $\Gamma$ is $(D-1)$-bounded.

We claim that $\gamma_{i}=1$ $(i\leq m)$ and $a_i=c_ia_1$ $(i\leq m)$. 

First, we consider the case $c_{m+1}=1$. 
Choose a vertex $x$ and a Delsarte clique $C$ with $d(x,C)=i$. 
As $\gamma_i\leq a_1+1$, we can choose a vertex $y\in C$ with $d(x,y)=i+1$. We find $C\cap\Gamma_i(x)\subseteq\Gamma(y)\cap\Gamma_i(x)$, i.e. $c_{i+1}\geq \gamma_i$ $(i\leq m)$. 
Note that $c_1\leq c_2\leq \cdots\leq c_D$ by \cite[Proposition 4.1.6]{BCN}. We have $\gamma_i=1$ $(i\leq m)$,  which implies $a_i=c_i a_1$ by Equation (\ref{eq:a4}). 

In the case $c_{m+1}\neq 1$, as $\Gamma$ is $m$-bounded, by Lemma \ref{hiraki}, we have $a_i= c_i a_1$  $(i\leq m)$. 
Note that $\gamma_0=1$. 
Assume $\gamma_{i-1}=1$ $(1\leq i\leq m)$, with $a_i=c_i a_1$, we have $\gamma_i=1$. 
By induction, it follows that $\gamma_i=1$ $(i\leq m)$.

As $E_D$ is a light tail with associated minimal idempotent $F$ corresponding to the eigenvalue $\theta'\neq k$ and $\theta_D=-\frac{k}{a_1+1}$, 
by \cite[Theorem 4.1]{JTZ10}, 
we have 
	\begin{align}
		u_i(\theta_D)^2=\alpha+\beta u_i(\theta')~~~~(0\leq i\leq D),\label{eq:a6}
	\end{align}
where $\theta'= \frac{k-(a_1+1)(a_1+2)}{(a_1+1)^2}$, $\alpha=\frac{a_1+1}{a_1 k +a_1 +1}$ and $\beta=\frac{a_1 k}{a_1k +a_1+1}$. 

For $i\leq m\leq D-2$, we have 
	\begin{align}
		u_{i+1}(\theta_D)&=(-\frac{1}{a_1+1})^{i+1}& &(0\leq i\leq m), \\
		u_{i+1}(\theta')&=\frac{(a_1+1)^{-2(i+1)}(a_1 k+a_1+1)-(a_1+1)}{ka_1}& &(0\leq i\leq m),\label{eq:a7}
	\end{align}
by Equations (\ref{eq:a3}) and (\ref{eq:a6}), as $\gamma_i =1$. 
Substitute  $a_i=c_i a_1$, $b_i=k-c_i(a_1+1)$ $(1\leq i\leq m)$ and Equation (\ref{eq:a7}) into 
    \begin{equation}	\label{leq}
    	c_i u_{i-1}(\theta')+a_iu_i(\theta')+b_i u_{i+1}(\theta')=\theta' u_i(\theta')~~~~(1\leq i\leq D),
    \end{equation}
    to obtain 
	\begin{align}
		c_i=\frac{(a_1+1)^{2i}-1}{(a_1+1)^2-1}~~~~(0\leq i\leq m).\label{eq:b1}
	\end{align}
It follows that $c_2=(a_1+1)^2+1\neq 1$. 
Since $\Gamma$ is $m$-bounded with $m\geq 2$ and $c_2\neq 1$, \cite[Proposition 19]{H09} implies $c_{m+1}\geq c_m + (c_m-c_{m-1})(c_2-c_1)>c_m$, i.e. 
	\begin{align}
		c_{m+1}\geq \frac{(a_1+1)^{2(m+1)}-1}{(a_1+1)^2-1}. \label{eq:b3}
	\end{align}

Equations (\ref{eq:a3}) and (\ref{eq:a6}) imply
	\begin{align}
		&u_{m+2}(\theta_D)=(-\frac{\gamma_{m+1}}{a_1+2-\gamma_{m+1}}) u_{m+1}(\theta_D), \label{eq:b2}\\
		&u_{m+2}(\theta')=\frac{\gamma_{m+1}^2(a_1+2-\gamma_{m+1})^{-2}(a_1+1)^{-2m-2}(a_1k+a_1+1)-(a_1+1)}{k a_1}. \label{eq:a8}
	\end{align} 

As $\Gamma$ is $2$-bounded and $c_2\neq 1$, Lemma \ref{hiraki} implies that $\Gamma$ contains an induced $GQ(a_1+1,c_2-1)$. 
As $\Gamma$ is not bipartite, we see $\theta_D\neq -k$ by \cite[Theorem 8.8.2]{AGT}. 
And $\theta_D=-\frac{k}{a_1+1}$ implies $a_1\neq 0$. 
It follows that $\theta'=\theta_1$ from Theorem \ref{lt}. 
Then we have $u_0(\theta')>u_1(\theta')>\cdots >u_D(\theta')$ by \cite[Theorem 13.2.1]{G93}. 
From Equations (\ref{eq:a6}) and (\ref{eq:b2}) , we see 
	\begin{align}
		\gamma_{m+1}\leq \frac{a_1+1}{2},\label{eq:b5}
	\end{align}
as $u_{m+1}(\theta')>u_{m+2}(\theta')$. 

Substitute Equations (\ref{eq:a5}), (\ref{eq:a7}) and (\ref{eq:a8}) into (\ref{leq}), we obtain 
	\begin{align}\label{eq:a9}
		c_{m+1}=\frac{(a_1+2-\gamma_{m+1})(a_1+1)((a_1+1)^{2(m+1)}-1)-k(a_1+2)(\gamma_{m+1}-1)}{(a_1+1)^2(a_1+2)(a_1+1-\gamma_{m+1})}. 
	\end{align} 

Substituting Equation (\ref{eq:a9}) into (\ref{eq:b3}), we see that 
\begin{align}\label{eq:c3}
	\frac{(\gamma_{m+1}-1)((a_1+1)((a_1+1)^{2m+2}-1)-((a_1+1)^2-1)k)}{a_1(a_1+1)^2(a_1+2)(a_1+1-\gamma_{m+1})}\geq 0.  
\end{align}

As $\gamma_m=1$, by Equation (\ref{eq:a4}), we see that $a_{m+1}\geq a_1c_{m+1}$. 
Since $m\leq D-2$, we see that $k>a_{m+1}+c_{m+1}\geq (a_1+1)c_{m+1}$ and Equation (\ref{eq:b3}) implies 
	\begin{align}
		k>(a_1+1)\frac{(a_1+1)^{2(m+1)}-1}{(a_1+1)^2-1}.\label{eq:b4}
	\end{align} 

We see $\gamma_{m+1}=1$ by Equations (\ref{eq:b5}) (\ref{eq:b4}) and (\ref{eq:c3}). 
It follows that $c_{m+1}=\frac{(a_1+1)^{2(m+1)}-1}{(a_1+1)^2-1}>c_m$ by Equations (\ref{eq:a9}) and (\ref{eq:b1}), and $a_{m+1}=c_{m+1}a_1$ by Equation (\ref{eq:a5}). 
This shows that $\Gamma$ is $(m+1)$-bounded by Lemma \ref{hiraki}. 	

Now we have that $\Gamma$ is $(D-1)$-bounded.  Note that $\gamma_{D-1}=1$. 
By the same way to obtain Equations (\ref{eq:a4}), (\ref{eq:a5}), and (\ref{eq:a7}), 
we have 
	\begin{align}
		a_D	&=(a_1+1)c_D-\frac{k}{a_1+1}\\
			&=a_1 c_D \label{eq:c1},\\
		u_D(\theta')&=\frac{(a_1+1)^{-2D}(a_1k+a_1+1)-(a_1+1)}{ka_1}.\label{eq:c2}
	\end{align}
Substitute $k=(a_1+1)c_D$, $a_D=c_Da_1$ and Equations (\ref{eq:a7}) (\ref{eq:c2}) into 
	\begin{align}
		c_D u_{D-1}(\theta')+a_D u_D(\theta')=\theta'u_D(\theta'),
	\end{align}
 to obtain that 
\begin{align}
	c_D=\frac{(a_1+1)^{2D}-1}{(a_1+1)^2-1}
\end{align}
holds.
Now we see that $\Gamma$ is a distance-regular graph with $a_i=c_i a_1$ and $c_i=\frac{(a_1+1)^{2i}-1}{(a_1+1)^2-1}~(i\leq D)$. 
As $\Gamma$ has diameter $D\geq 3$, by \cite[Theorem 9.4.7]{BCN}, $\Gamma$ is $^2A_{2D-1}(r)$ with $r=a_1+1$. 
\qed

\noindent
{\bf Remark.}
The Hermitian dual polar graph $^2A_{2D-1}(r)$ is determined by its intersection array for diameter $D\geq 3$ (see \cite[Theorem 9.4.7]{BCN}). 
The graph $^2A_3(r)$ is a generalized quadrangle $GQ(r,r^2)$. It is known that for $r$ a prime power with $r\equiv 2 (\mod ~3)$ and $r>2$, there are at least two non-isomorphic generalized quadrangles $GQ(r,r^2)$'s, so $^2A_3(r)$ is not determined by its intersection array for those prime powers  (see \cite[3.1.2, 3.1.6 and 3.2.5]{PT09}). 
\\
\\
\section{Proofs of Theorem 1.2}

First we establish the following consequence of Theorem 1.1.
\begin{corollary}	\label{c2dpl}
Let $\Gamma$ be a non-bipartite geometric distance-regular graph with valency $k\geq 3$, diameter $D\geq 3$ and eigenvalues $\theta_0>\theta_D>\cdots>\theta_D$. 
If $\Gamma$ is $m$-bounded for some $m\geq 2$ and $c_2 \geq (a_1+1)^2+1$, then $c_2 = (a_1 +1)^2+1$ and $\Gamma$ is the dual polar graph $^2A_{2D-1}(r)$ with $r=a_1+1$. 
\end{corollary}
\begin{proof}
The graph $\Gamma$ is non-bipartite geometric implies that $\theta_D=-\frac{k}{a_1+1}$ and $a_1\neq 0$. 
As $\Gamma$ is $2$-bounded with $c_2\geq (a_1+1)^2+1\neq 1$, by Lemma \ref{hiraki}, we see that $\Gamma$ contains an induced subgraph $GQ(a_1+1,c_2-1)$  and hence $\theta_1\leq \frac{k-(a_1+1)}{c_2-1}-1$, by Lemma \ref{ubd}. Lemma \ref{lbd} implies that $c_2 = (a_1+1)^2+1$ and now it follows that
$E_D$ is a light tail with associated minimal idempotent $E_1$, by Theorem \ref{lt}. 
Now, the result follows from Theorem \ref{dlp}. 
\end{proof}

{\bf Proof of Theorem 1.2.}
Every triangle of $\Gamma$ is a Delsarte clique and $\Gamma$ is geometric. 

Choose a vertex $x$ and a Delsarte clique $C$ with $d(x,C)=i$. Then there exists a vertex $y\in\Gamma(x)$ such that $d(y,C)=i-1$. 
Note that $\Gamma_{i-1}(y)\cap C\subseteq \Gamma_i(x)\cap C$, which implies $\gamma_i\geq \gamma_{i-1}$ $(1\leq i\leq D)$.  
So the sequence $\gamma_i$ is non-decreasing and $1\leq \gamma_i\leq a_1+1=2$.

Then there exists some integer $e$ such that
\begin{equation}
    \gamma_i=
        \left\{
        \begin{aligned}
            &1,~ i\leq e-1,\\
            &2,~ i\geq e.
        \end{aligned}
        \right.
\end{equation}
And (\ref{eq:f}) implies the following 
\begin{equation}
    u_i(\theta_D)=
        \left\{
        \begin{aligned}
            &(-\frac{1}{2})^i,~i\leq e,\\
            &(-\frac{1}{2})^{2e-i},~i\geq e+1. 
        \end{aligned}
        \right.
\end{equation}
As $|u_i(\theta_D)|\leq 1$ , which implies $e\geq \frac{D}{2}$. 
If $D\geq 5$, then $e\geq 3$. 
Then $a_i=c_ia_1$ $(i\leq e-1)$ by Equation (\ref{eq:a4}). 
Lemma \ref{hiraki} implies that $\Gamma$ is $2$-bounded, hence it follows that $c_2 =5$ and  $\Gamma$ is the dual polar graph $^2A_{2D-1}(2)$ by Corollary \ref{c2dpl}. 
The case when $2\leq D\leq 4$ follows from \cite[Theorem 1.2]{KQ15}.
\qed

{\bf Acknowledgements.}
We would like to thank Jongyook Park. He carefully read an earlier version of the note and his comments improved the note significantly.

\bibliographystyle{plain}
\bibliography{bib}

\begin{thebibliography}{10}

\bibitem{BKP15}
S.~Bang, J.H. Koolen, and J.~Park.
\newblock Some results on the eigenvalues of distance-regular graphs.
\newblock {\em Graphs and Combin.}, 2015.

\bibitem{List}
A.E. Brouwer.
\newblock Parameters of strongly regular graphs.
\newblock \url{http://www.win.tue.nl/~aeb/graphs/srg/srgtab.html}, September
  2015.

\bibitem{BCN}
A.E. Brouwer, A.M. Cohen, and A.~Neumaier.
\newblock {\em Distance-Regular Graphs}.
\newblock Springer-Verlag, Berlin, 1989.

\bibitem{D73}
P.~Delsarte.
\newblock {\em An algebraic approach to the association schemes of coding
  theory}, volume~10 of {\em Philips Res. Reports Suppl.}
\newblock Philips, 1973.

\bibitem{G93}
C.D. Godsil.
\newblock {\em Algebraic Combinatorics}.
\newblock Chapman \& Hall, New York, 1993.

\bibitem{AGT}
C.D. Godsil and G.~Royle.
\newblock {\em Algebraic Garph Theory}, volume 207 of {\em GTM}.
\newblock Springer, New York, 2001.

\bibitem{H99}
A.~Hiraki.
\newblock Strongly closed subgraphs in a regular thick near polygon.
\newblock {\em European J. Combin.}, 20:789--796, 1999.

\bibitem{H09}
A.~Hiraki.
\newblock A characterization of some distance-regular graphs by strongly closed
  subgraphs.
\newblock {\em European J. Combin.}, 30:893--907, 2009.

\bibitem{JTZ10}
A.~Juri\v{s}i\'{c}, P.~Terwilliger, and A.~\v{Z}itink.
\newblock Distance-regular graphs with light tails.
\newblock {\em European J. Combin.}, 31:1539--1552, 2010.

\bibitem{KQ15}
J.H. Koolen and Z.~Qiao.
\newblock Distance-regular graph with valency $k$ and smallest eigenvalue at
  most $-\frac{k}{2}$.
\newblock [arXiv:1507.04839v2].

\bibitem{P99}
A.A. Pascasio.
\newblock Tight graphs and their primitive idempotents.
\newblock {\em J. Algebraic Combin.}, 10:47--59, 1999.

\bibitem{PT09}
S.E. Payne and J.A. Thas.
\newblock {\em Finite Generalized Quadrangles}.
\newblock European Mathematical Society, second edition, 2009.

\bibitem{DKT}
E.R. van Dam, J.H. Koolen, and H.~Tanaka.
\newblock Distance-regular graphs.
\newblock [arXiv:1410.6294v1].

\end{thebibliography}
\nocite{*}

\end{document}